\documentclass[reqno,b5paper]{amsart}
\usepackage{amsmath}
\usepackage{amssymb}
\usepackage{amsthm}
\usepackage{enumerate}
\usepackage[mathscr]{eucal}
\usepackage{eqlist}

\theoremstyle{plain}
\newtheorem{theorem}{Theorem}[section]
\newtheorem{prop}[theorem]{Proposition}
\newtheorem{lemma}{Lemma}[section]

\theoremstyle{definition}
\newtheorem{definition}{Definition}[section]

\theoremstyle{remark}

\numberwithin{equation}{section}

\begin{document}

\title[Report on the absolute differential equations I]%
{Report on the absolute differential equations I}

\author[Veronika Chrastinov\'a \and V\'aclav Tryhuk]%
{Veronika Chrastinov\'a* \and V\'aclav Tryhuk**}

\newcommand{\acr}{\newline\indent}

\address{Brno University of Technology\acr
Faculty of Civil Engineering\acr
\llap{*\,}Department of Mathematics\acr
\llap{**\,}AdMaS Center\acr
Veve\v{r}\'{\i} 331/95, 602 00 Brno\acr
Czech Republic}
\email{chrastinova.v@fce.vutbr.cz,tryhuk.v@fce.vutbr.cz,tryhuk.v@outlook.com}

\thanks{This paper was elaborated with the financial support of the European
Union's "Operational Programme Research and Development for
Innovations", No. CZ.1.05/2.1.00/03.0097, as an activity of the
regional Centre AdMaS "Advanced Materials, Structures and
Technologies".}

\subjclass[2010]{35-02, 35A30, 58A20, 58J70}

\keywords{higher--order transformations, symmetry, diffiety, involutivity, controllability, characteristics}

\begin{abstract}
The article provides a~modest survey of the absolute theory of general systems of (partial) differential equations. The equations are relieved of all additional structures and subject to quite arbitrary change of the variables. An~abstract mathematical theory in the Bourbaki sense with its own concepts and technical tools follows. In particular the external, internal, generalized and higher--order symmetries and infinitesimal symmetries together with the E. Cartan's prolongations, various characteristics, the involutivity and the controllability structures are clarified in genuinely coordinate--free terms without any use of the common jet mechanisms.
\end{abstract}

\maketitle

\section{Preface}\label{sec1}
We motivate the central concepts and informally describe the main task in order to give some impressions on our subject, however, this is not a~logical prerequisite for the text to follow.
\subsection{The higher--order transformations (the morphisms)}\label{ssec1.1}
Let us recall the \emph{jet coordinates}
\begin{equation}\label{eq1.1}
x_i,w^j_I\quad (j=1,\ldots,m;\, I=i_1\cdots i_r;\,i,i_1,\ldots,i_r=1,\ldots,n;\,r=0,1,\ldots\,)
\end{equation}
where $x_i$ are \emph{independent variables}, $w^j$ (empty $I$) \emph{dependent variables} and $w^j_I$ (nonempty $I$) correspond to the \emph{derivatives}
\[w^j_I=\frac{\partial w^j}{\partial x_I}=\frac{\partial^rw^j}{\partial x_{i_1}\cdots\partial x_{i_r}}\quad (I=i_1\cdots i_r).\]
Let us moreover introduce the equations
\begin{equation}\label{eq1.2}\bar x_i=X_i(\cdot\cdot,x_{i'},w^{j'}_{I'},\cdot\cdot), \bar w^j=W^j(\cdot\cdot,x_{i'},w^{j'}_{I'},\cdot\cdot)\quad (i=1,\ldots,n;\,j=1,\ldots,m)\end{equation}
where $X_i$ and $W^j$ are given functions of a~finite number of variables (\ref{eq1.1}). They are interpreted as the transformation formulae: the functions \[w^j=w^j(x_1,\ldots,x_n);\ j=1,\ldots,m;\] are transformed into certain functions \[\bar w^j=\bar w^j(\bar x_1,\ldots,\bar x_n);\ j=1,\ldots,m;\] and this is made as follows. Denoting
\[\mathcal X_i=X_i(\cdot\cdot,x_{i'},\frac{\partial w^{j'}}{\partial x_{I'}}(x_1,\ldots,x_n),\cdot\cdot)=\mathcal X_i(x_1,\ldots,x_n)\quad (i=1,\ldots,n),\]
we suppose that\[\det\left(\frac{\partial\mathcal X_{i'}}{\partial x_i}\right)=\det(D_i\mathcal X_{i'})\neq 0\qquad (D_i=\frac{\partial}{\partial x_i}+\sum w^j_{Ii}\frac{\partial}{\partial w^j_I}).\]
Then the implicit system $\bar x_i=\mathcal X_i(x_1,\ldots,x_n); i=1,\ldots,n;$ admits a~solution $x_i=\bar{\mathcal X_i}(\bar x_1,\ldots,\bar x_n); i=1,\ldots,n;$ and provides the desired result
\[\bar w^j=W^j(\cdot\cdot,\bar{\mathcal X_i},\frac{\partial w^{j'}}{\partial x_{I'}}(\bar{\mathcal X_1},\ldots,\bar{\mathcal X_n}),\cdot\cdot)=\bar w^j(\bar x_1,\ldots,\bar x_n)\qquad (j=1,\ldots,m).\]
One can also obtain the transformation of the derivatives
\begin{equation}\label{eq1.3}\bar w^j_I=\frac{\partial\bar w^j}{\partial\bar x_I}=W^j_I(\cdot\cdot,x_{i'},w^{j'}_{I'},\cdot\cdot)\qquad (\text{all }j\text{ and }I)\end{equation}
which complete the equations (\ref{eq1.2}). They satisfy the recurrence
\begin{equation}\label{eq1.4} \sum W^j_{Ii}D_iX_{i'}=D_{i'}W^j_I\qquad (\text{all }i',j,I)\end{equation}
and we altogether speak of a~\emph{morphism} (\ref{eq1.2}), (\ref{eq1.3}). If there exists the inverse morphism
\begin{equation}\label{eq1.5}x_i=\bar X_i(\cdot\cdot,\bar x_{i'},\bar w^{j'}_{I'},\cdot\cdot), w^j_I=\bar W^j_I(\cdot\cdot,\bar x_{i'},\bar w^{j'}_{I'},\cdot\cdot),\end{equation}
we speak of an~\emph{automorphism} (or: \emph{symmetry}). The totality of all symmetries unexpectedly manifests as an unheard--of mystery \cite{T2, T1}. Just the symmetries are important: they produce the higher--order equivalences of differential equations.

\subsection{Example: the wave construction}\label{ssec1.2}
Assuming $n=1,$ we abbreviate
\[x=x_1, \bar x=\bar x_1, w^j_r=w^j_{1\cdots 1}, \bar w^j_r=\bar w^j_{1\cdots 1}\qquad (r\text{ terms}).\]
Let $V=V(x,w^1_0,\ldots,w^m_0,\bar x,\bar w^1_0,\ldots,\bar w^m_0)$ be a~given function.
\\\\
{\bf Proposition} \cite{T4, T3}. \emph{If the implicit system
\[V=DV=\cdots =D^mV=0\qquad(D=\frac{\partial}{\partial x}+\sum w^j_{r+1}\frac{\partial}{\partial w^j_r})\]
admits a~solution
\begin{equation}\label{eq1.6}
\bar x=X(x,w^1_0,\ldots,w^m_m),\ \bar w^j_0=w^j_0(x,w^1_0,\ldots,w^m_m)\quad (j=1,\ldots,m)
\end{equation} such that $DX\neq 0$ and the implicit system
\[V=\bar DV=\cdots =\bar D^mV=0\qquad(\bar D=\frac{\partial}{\partial\bar x}+\sum \bar w^j_{r+1}\frac{\partial}{\partial\bar w^j_r})\] admits a~solution
\begin{equation}\label{eq1.7}
x=\bar X(\bar x,\bar w^1_0,\ldots,\bar w^m_m),\ w^j_0=\bar w^j_0(\bar x,\bar w^1_0,\ldots,\bar w^m_m)\quad (j=1,\ldots,m)
\end{equation} such that $\bar D\bar X\neq 0$ then $(\ref{eq1.6})$ and $(\ref{eq1.7})$ are symmetries inverse of each other.}

\bigskip

The Proposition can be generalized for the case $n>1$ and all the classical Lie contact transformations are involved if $m=1.$ Except for this Lie's favourable subcase with $m=1$ and arbitrary $n,$ all such symmetries destroy the finite--order jet spaces.

\subsection{The infinitesimal modification \cite{T6, T5, T7}}\label{ssec1.3}
The ancient ``linear approximation'' of equations (\ref{eq1.2}) and (\ref{eq1.3}) reads
\[\bar x_i=x_i+\varepsilon z_i(\cdot\cdot,x_{i'},w^{j'}_{I'},\cdot\cdot),\ \bar w^j_I=w^j_I+\varepsilon z^j_I(\cdot\cdot,x_{i'},w^{j'}_{I'},\cdot\cdot)\] where $\varepsilon$ is the famed ``small parameter''. In the rigorous theory, let us instead introduce the vector field
\begin{equation}\label{eq1.8}
Z=\sum z_i\frac{\partial}{\partial x_i}+\sum z^j_I\frac{\partial}{\partial w^j_I}\qquad (z^j_{Ii}=D_iz^j_I-\sum w^j_{Ii'}D_{i'}z_i)
\end{equation}
with the recurrence following from (\ref{eq1.4}).

\medskip

{\sl Warning}. In contrast to actual convention, we speak of a~\emph{variation} and the common term the \emph{generalized} (or: \emph{higher--order}, or: \emph{Lie--B\H{a}cklund}) \emph{infinitesimal transformation} (briefly: the \emph{symmetry}) is retained only for the case when the vector field $Z$ generates a~true (local) Lie group. The totality of all infinitesimal symmetries is unknown. 

\subsection{Example: variations and symmetries}\label{ssec1.4}
We again suppose $n=1.$ The vector field
\[D=Z=\frac{\partial}{\partial x}+\sum w^j_{r+1}\frac{\partial}{\partial w^j_r}\qquad (z_1=1,\, z^j_{1\cdots 1}=z^j_r=w^j_{r+1}\text{ with }r\text{ terms})\]
is clearly a~variation but not a~symmetry since the Lie system
\[\frac{\partial X}{\partial t}=z_1=1,\ \frac{\partial W^j_r}{\partial t}=W^j_{r+1}=DW^j_r\qquad (X|_{t=0}=x,W^j_r|_{t=0}=w^j_r)\]
for the corresponding Lie group
\[\bar x(t)=X(\cdot\cdot,x,w^{j'}_{r'},\cdot\cdot;t),\ \bar w^j_r(t)=W^j_r(\cdot\cdot,x,w^{j'}_{r'},\cdot\cdot;t)\quad (-\varepsilon<t<\varepsilon)\]
is contradictory. On the contrary, the vector field
\[Z=\sum w^1_{r+1}\frac{\partial}{\partial w^k_r}\qquad (m\geq 2; \text{the sum over } k=2,\ldots,m \text{ and }r=0,1,\ldots\;)\]
generates the very simple ``higher--order Lie group'' of the morphisms
\[\bar x(t)=x, \bar w^1_r(t)=w^1_r, \bar w^k_r(t)=w^k_r+tw^1_{r+1}\quad (k=2,\ldots, m;\,r=0,1,\ldots).\]
This group does not preserve many of the classical concepts, even the order of the differential equations.

\subsection{External differential equations \cite{T7}}\label{ssec1.5}
We denote $\mathbf M(m,n)$ the space equipped with coordinates (\ref{eq1.1}). Differential equations are traditionally interpreted as the subspace $\mathbf M\subset\mathbf M(m,n)$ defined by certain equations
\begin{equation}\label{eq1.9}D_{i_1}\cdots D_{i_r}f^k=0\qquad (k=1,\ldots,K;\,i_1,\ldots,i_r=1,\ldots,n;\,r=0,1,\ldots)\end{equation}
where $f^1,\ldots,f^K$ are given functions of variables (\ref{eq1.1}).

Let us recall the morphisms.  By using the more precise pull--back notation
\[\mathbf m^*x_i=\bar x_i, \mathbf m^*w^j_I=\bar w^j_I\] in the equations (\ref{eq1.2}) and (\ref{eq1.3}), they may be interpreted as a~mapping $\mathbf m: \mathbf M(m,n)\rightarrow\mathbf M(m,n),$ the \emph{morphism of the space} $\mathbf M(m,n).$
Assuming moreover $\mathbf m\mathbf M\subset\mathbf M,$ then $\mathbf m$ is said to be the \emph{external morphism} of $\mathbf M$ and the \emph{external automorphism} (or: \emph{external symmetry}) of $\mathbf M$ in the invertible case (\ref{eq1.5}).

The infinitesimal concepts are analogous: if the vector field (\ref{eq1.8}) is tangent to $\mathbf M$ (and therefore $Z$ is a~vector field on $\mathbf M$ as well) then $Z$ is called the \emph{external variation} of $\mathbf M$ and if $Z$ is moreover a~symmetry, we have the \emph{external infinitesimal symmetry} of $\mathbf M.$

\subsection{Internal differential equations \cite{T7}}\label{ssec1.6}
The external morphism of $\mathbf M$ whose restriction to $\mathbf M$ is invertible is the \emph{internal} symmetry. The external variation $Z$ of $\mathbf M$ which generates a~Lie group on $\mathbf M$ is the \emph{internal} infinitesimal symmetry. These internal concepts are in fact independent of the localization $\mathbf M$ in the ambient space $\mathbf M(m,n).$ The reasons are as follows.

Let $\Omega(m,n)$ be the module of all \emph{contact forms}
\begin{equation}\label{eq1.10} \omega=\sum f^j_I\omega^j_I\qquad (\text{finite sum}, \omega^j_I=dw^j_I-\sum w^j_{Ii}dx_i) \end{equation}
on the space $\mathbf M(m,n).$ One can observe that the recurrence (\ref{eq1.3}) is equivalent to the inclusion $\mathbf m^*\Omega(m,n)\subset\Omega(m,n).$ Analogously the recurrence in  (\ref{eq1.8}) is expressed by the inclusion $\mathcal L_Z\Omega(m,n)\subset\Omega(m,n)$ for the Lie derivative.

Let $\Omega$ be the restriction of the module $\Omega(m,n)$ to the subspace $\mathbf M\subset\mathbf M(m,n).$ The morphism $\mathbf m$ restricted to $\mathbf M$ clearly satisfies $\mathbf m^*\Omega\subset\Omega.$ Analogously the variation $Z$ restricted to $\mathbf M$ also satisfies the inclusion $\mathcal L_Z\Omega\subset\Omega.$ Therefore both concepts are characterized without the use of the ambient space $\mathbf M(m,n).$ However more is true: \emph{even the module $\Omega$ itself} can be characterized in abstract terms, we speak of a~\emph{diffiety} $\Omega$ on $\mathbf M,$ see below.

Altogether we obtain the \emph{internal} theory on $\mathbf M$ not affected by the space $\mathbf M(m,n).$

\subsection{Example: the internal symmetry}\label{ssec1.7}
Assuming $n=1,$ we introduce the subspace $\mathbf M\subset\mathbf M(m,1)$ defined by the equations
\[D^rw^j_2=w^j_{r+2}=0\qquad (j=1,\ldots,m;\,r=0,1,\ldots).\]
Then the morphism $\mathbf m:\mathbf M(m,1)\rightarrow\mathbf M(m,1)$ where
\[\mathbf m^*x=\bar x=x, \mathbf m^*w^j_r=\bar w^j_r=w^j_r+w^j_{r+1}\quad (j=1,\ldots,m;\,r=0,1,\ldots)\]
is clearly noninvertible but the restriction to $\mathbf M$
\[\mathbf m^*x=x, \mathbf m^*w^j_0=w^j_0+w^j_1, \mathbf m^*w^j_1=w^j_1\quad (j=1,\ldots,m)\]
is a~symmetry. Analogously the vector field $Z=D$ is a~mere variation on $\mathbf M(m,1)$ but generates the Lie group
\[\bar x(t)=x, \bar w^j_0(t)=w^j_0+tw^j_1, \bar w^j_1(t)=w^j_1\quad (j=1,\ldots,m)\]
on the space $\mathbf M.$ So we have the internal but not the external symmetries.

\subsection{Use of the Pfaffian equations}\label{ssec1.8}
For instance, let us mention the equation \[dz=pdx+F(x,y,z,p)dy\] where $x,y,z,p$ may (or may not) be regarded as the coordinates of the underlying space. The solutions $z=z(x,y)$ parametrized with $x,y$ satisfy \[\frac{\partial z}{\partial y}=F(x,y,z,\frac{\partial z}{\partial x}).\]
However \emph{the same} Pfaffian equation \[d\bar z=-xdp+F(x,y,\bar z+px,p)dy\quad (\bar z=z-px)\] admits the solutions $\bar z=\bar z(p,y)$ which satisfy \[\frac{\partial\bar z}{\partial y}=F(-\frac{\partial\bar z}{\partial p},y,\bar z-p\frac{\partial\bar z}{\partial p},p).\]
And quite analogously, \emph{the same} Pfaffian equation \[dy=\frac{1}{F}dz-\frac{p}{F}dx\quad (F\neq 0, F_p\neq 0)\] admits the solutions $y=y(z,x)$ which satisfy \[\frac{\partial y}{\partial x}=-p\frac{\partial y}{\partial z}\quad (F(x,y,z,p)=1/\frac{\partial y}{\partial z}\text{ determines }p=p(x,y,z,\frac{\partial y}{\partial z})).\]

{\sl We conclude.} A~Pfaffian equation represents many formally quite dissimilar but in fact equivalent differential equations according to the additional choice of the dependent and the independent variables. It follows that a~coordinate--free theory should be expressed in terms of the Pfaffian equations.

\subsection{Towards the diffieties}\label{ssec1.9}
Let us finally recall the subspace $\mathbf i:\mathbf M\subset\mathbf M(m,n).$ The definition equations (\ref{eq1.9}) imply that vector fields $D_1,\ldots, D_n$ are tangent to the subspace $\mathbf M$ and therefore may be regarded as vector fields on $\mathbf M$ as well. They satisfy the crucial identity
\[\mathcal L_{D_{i'}}\omega^j_I=\mathcal L_{D_{i'}}(dw^j_I-\sum w^j_{Ii}dx_i)=dw^j_{Ii'}-\sum w^j_{Iii'}dx_i=\omega^j_{Ii'}\]
which is clearly true even for the restrictions $\mathbf i^*\omega^j_I$ of the forms $\omega^j_I$ to the space $\mathbf M.$ Let $\Omega(m,n)_l\subset\Omega(m,n)$ be the submodule of all contact forms of the order $l$ at most, hence $|I|=r\leq l$ in all summands (\ref{eq1.9}). Obviously
\begin{equation}\label{eq1.11}\Omega(m,n)_l+\sum\mathcal L_{D_{i'}}\Omega(m,n)_l=\Omega(m,n)_{l+1}\qquad (\text{all }l)\end{equation}
and this property is true even for the restrictions to $\mathbf M$, i.e., for the submodules 
\begin{equation}\label{eq1.12}\Omega_l=\mathbf i^*\Omega(m,n)_l\subset\mathbf i^*\Omega(m,n)=\Omega.\end{equation}
We have in fact discovered the crucial property of the diffieties~$\Omega.$

\medskip

{\sl We conclude.} Differential equations should be represented by the Pfaffian system $\omega=0$ $(\omega\in\Omega)$ on the space $\mathbf M$ where the modules $\Omega$ are described in abstract terms (the \emph{internal theory}). Then the actual choice of the dependent and the independent variables is irrelevant (the \emph{absolute approach}).

\section{Fundamental concepts}\label{sec2}
We deal with smooth and local category of manifolds and mappings. Our notational convention for a~mapping $\mathbf m:\mathbf M\rightarrow\bar{\mathbf M}$ of manifolds allows the definition domain to be a~proper open subset of $\mathbf M.$ In order to delete the ``singular points,'' we moreover tacitly suppose the existence of bases in all modules to appear. Then the parade of the primary concepts denoted (I)--(VI) looks as follows.

\subsection*{(I) On the manifolds \cite{T6}}\label{ssec2.1}
Besides the occasional use of the common finite--dimensional spaces, we mainly deal with manifolds $\mathbf M$ modelled on $\mathbb R^\infty,$ i.e., there are coordinates $h^j:\mathbf M\rightarrow\mathbb R$ $(j=1,2,\ldots)$ such that the ring $\mathcal F\ (=\mathcal F(\mathbf M),$ the abbreviation whenever possible) of admissible functions $f:\mathbf M\rightarrow\mathbb R$ involves just the (smooth) composite functions $f=F(h^1,\ldots,h^{m(f)}).$
Then the $\mathcal F$--module $\Phi$ $(=\Phi(\mathbf M))$ of differential forms $\varphi=\sum f^jdg^j$ (finite sum with $f^j,g^j\in\mathcal F$) and the $\mathcal F$--module $\mathcal T\ (=\mathcal T(\mathbf M))$ of vector fields $Z$ on the space $\mathbf M$ make a~good sense. It should be noted that the vector fields are regarded as $\mathcal F$--linear functions $Z:\Phi\rightarrow\mathcal F$ where
\[df(Z)=Zf,\ \varphi(Z)=Z\rfloor\varphi=\sum f^jZg^j\]
and if $\varphi^1,\varphi^2,\ldots$ is a~basis of module $\Phi,$ we denote
\[Z=\sum z^j\frac{\partial}{\partial\varphi^j}\quad (\text{infinite series, arbitrary }z^j=\varphi^j(Z)\in\mathcal F)\]
with the common abbreviation $\partial/\partial f=\partial/\partial df.$ The familiar rules like
\[\mathcal L_Zf=Zf,\ \mathcal L_Z\varphi=Z\rfloor d\varphi+d\varphi(Z),\ \mathcal L_ZX=[Z,X],\ \mathcal L_{[X,Y]}=\mathcal L_X\mathcal L_Y-\mathcal L_Y\mathcal L_X\]
for the Lie derivative and the Lie bracket do not need any comment.

Let $\mathbf n:\mathbf N\rightarrow\mathbf M$ be a~mapping of manifolds. If an~appropriate part of the family of functions $\mathbf n^*h^1,\mathbf n^*h^2,\ldots$ can be taken for the coordinates on $\mathbf N,$ then $\mathbf n$ is called an \emph{inclusion} of the \emph{submanifold} $\mathbf N$ into the space $\mathbf M.$
(Since $\mathbf n$ is injective, we occasionally identify $\mathbf N=\mathbf n\mathbf N\subset\mathbf M$ with the subset of $\mathbf M.$)
Analogously $\mathbf n$ is called a~\emph{projection} of $\mathbf N$ on the \emph{factorspace} $\mathbf M$ if the family $\mathbf n^*h^1,\mathbf n^*h^2,\ldots$ can be completed by some additional functions to the coordinates on $\mathbf N.$
(Since $\mathbf n^*$ is injective, we occasionally abbreviate $f=\mathbf n^*f,\varphi=\mathbf n^*\varphi.$)

\subsection*{(II) On the diffieties \cite{T6, T8}}\label{ssec2.2}
For every submodule $\Omega\subset\Phi,$ let $\mathcal H(\Omega)\subset\mathcal T$ be the submodule of all vector fields $Z$ such that $\Omega(Z)=0.$ A~submodule $\Omega\subset\Phi$ is called \emph{flat} (or: \emph{satisfying the Frobenius condition}) if any of the equivalent requirements
\begin{equation}\label{eq2.1} d\Omega\cong 0\ (\text{mod }\Omega),\ \mathcal L_\mathcal H\Omega\subset\Omega,\ [\mathcal H,\mathcal H]\subset\mathcal H\qquad (\mathcal H=\mathcal H(\Omega))\end{equation} is satisfied. The finite--dimensional flat submodules are simple: they admit a~basis consisting of total differentials. We are however interested just in the infinite--dimensional case.
\begin{definition}\label{def2.1} A~finite--codimensional submodule $\Omega\subset\Phi$ is called a~\emph{diffiety} if there exists a~\emph{good filtration} $\Omega_*:\Omega_0\subset\Omega_1\subset\cdots\subset\Omega=\cup\Omega_l$ with the finite--dimensional submodules $\Omega_l\subset\Omega$ $(l=0,1,\ldots)$ such that
\begin{equation}\label{eq2.2}\mathcal L_\mathcal H\Omega_l\subset\Omega_{l+1}\quad (\text{all }l),\quad \Omega_l+\mathcal L_\mathcal H\Omega_l=\Omega_{l+1}\quad (l \text{ large enough}).\end{equation}
A~\emph{pre--diffiety} need not satisfy the codimensionality requirement. \end{definition}
We deal only with the diffieties unless otherwise stated. The pre--diffieties will appear later on and will be reduced to the common diffieties.
\begin{definition}\label{def2.2} Denoting $n\ (=n(\Omega))=\dim\Phi/\Omega=\dim\mathcal H,$ functions $x_1,\ldots,x_n$ are called \emph{independent variables} of diffiety $\Omega$ if the differentials $dx_1,\ldots,dx_n$ are linearly independent $\text{mod}\,\Omega.$ Then the \emph{total derivatives} $D_1,\ldots,D_n\in\mathcal H$ defined by
\[D_ix_i=1,\ D_{i'}x_i=0,\ \Omega(D_i)=0\qquad (i,i'=1,\ldots,n;\,i\neq i')\]
constitute a~basis of module $\mathcal H$ and the \emph{contact forms}
\[\omega\{f\}=df-\sum D_ifdx_i\quad (f\in\mathcal F)\]
generate the diffiety $\Omega.$ \end{definition}
The link to the classical approach can be succintly described as follows.

\medskip

Let us consider differential equations $\mathbf i:\mathbf M\subset\mathbf M(m,n)$ in the sense (\ref{eq1.9}). It was already clarified in Preface~\ref{ssec1.9} that the restriction $\Omega=\mathbf i^*\Omega(m,n)$ of the module $\Omega(m,n)$ to the space $\mathbf M$ is a~diffiety with the terms $\Omega_l=\mathbf i^*\Omega(m,n)_l$ of the good filtration. Let us conversely start with a~diffiety~$\Omega.$ Due to (\ref{eq2.2}), there exist generators
\[\mathcal L_{D_{i_1}}\cdots\mathcal L_{D_{i_r}}\omega^k\quad (k=1,\ldots,K;\,i_1,\ldots,i_r=1,\ldots,n;\,r=0,1,\ldots\,)\]
of module $\Omega.$ Denoting $\omega^k=\sum a^k_jdh^j,$ the Pfaffian system $\omega=0\ (\omega\in\Omega)$ clearly reads
\[\sum f^k_j\frac{\partial h^j}{\partial x_i}=0,\ D_{i_1}\sum f^k_j\frac{\partial h^j}{\partial x_i}=0,\ D_{i_1}D_{i_2}\sum f^k_j\frac{\partial h^j}{\partial x_i}=0,\ \ldots\]
which is a~classical (infinitely prolonged) system of differential equations for a~\emph{finite number} of functions $h^j=h^j(x_1,\ldots,x_n)$ occuring in the forms $\omega^1,\ldots,\omega^K.$

\medskip

{\sl We conclude.} A~diffiety supplied with a~fixed choice of the dependent and independent variables is the same as an~infinitely prolonged system of differential equations.

\subsection*{(III) On the commutative algebra \cite{T6, T8, T9, T10, T7}}\label{ssec2.3}
Conditions (\ref{eq2.2}) simplify if the original filtration $\Omega_*$ is replaced with the gradation
\[\mathcal M=\mathcal M_0\oplus\mathcal M_1\oplus\cdots\qquad(\mathcal M_l=\Omega_l/\Omega_{l-1},\ \Omega_{-1}=0).\]
Then the Lie derivative $\mathcal L_D$ turns into the $\mathcal F$--linear mapping $D:\mathcal M\rightarrow\mathcal M$ such that
\[D[\omega]=[\mathcal L_D\omega]\in\mathcal M_{l+1}\qquad (\omega\in\Omega_l, [\omega]\in\mathcal M_l, D\in\mathcal H)\]
where the square brackets denote the factorization. But more is true. Let
\[\mathcal A\ (=\mathcal A(\mathcal H))=\mathcal A_0\oplus\mathcal A_1\oplus\cdots\qquad (\mathcal A_0=\mathcal F, A_1=\mathcal H, \mathcal A_2=\mathcal H\odot\mathcal H,\ldots)\]
be the algebra of homogeneous polynomials over $\mathcal H.$ We obtain even the $\mathcal A$--module structure on $\mathcal M.$  For instance
\[D_{i_1}\cdots D_{i_r}[\omega]=[\mathcal L_{D_{i_1}}\cdots\mathcal L_{D_{i_r}}\omega]\in\mathcal M_{l+r}\qquad (\omega\in\Omega_r, D_{i_1}\cdots D_{i_r}\in\mathcal A_r),\]
\[D_{i_1}\cdots D_{i_r}[\omega\{f\}]=[\omega\{D_{i_1}\cdots D_{i_r}f\}].\]

\medskip

{\sl Warning.} The algebraical calculations with the $\mathcal F$--module are performed at a~fixed point of $\mathbf M.$ So we deal with ``smooth families'' of $\mathbb R$--modules. It follows that the classical algebra can be applied.

\bigskip

In particular we recall the \emph{Hilbert polynomial}
\[\dim\mathcal M_l=e_\nu\binom l\nu+\cdots+e_0\binom l0\qquad (l\text{ large enough, } e_\nu\neq 0)\]
with integer coefficients, alternatively
\[\dim\Omega_l=e_\nu\binom{l+1}{\nu+1}+\cdots +e_{-1}\binom{l+1}0\qquad  (l\text{ large enough, } e_\nu\neq 0).\]
Then the degree $\nu\ (=\nu(\Omega))$ and the integer $\mu\ (=\mu(\Omega))=e_\nu>0$ do not depend on the choice of the filtration $\Omega_*.$
\begin{definition}\label{def2.3} We claim that \emph{the solution} of diffiety $\Omega$ \emph{depends} on $\mu(\Omega)$ functions of $\nu(\Omega)+1$ variables. Diffiety $\Omega$ is \emph{overdetermined, determined, or underdetermined} according to whether $\nu+1<n-1, \nu+1=n-1,$ or $\nu+1>n-1,$ respectively. \end{definition}
One can observe that we have not yet introduced the concept of a~solution of diffiety~$\Omega.$ Definition~\ref{def2.3} therefore designates a~more \emph{formal property} of diffiety~$\Omega,$ however, it is in full accordance with quite opposite approach~\cite{T12}--\cite{T17} in the theory of exterior differential systems.

\medskip

{\sl Remark.} The above reasonings make sense even for the much easier case of the finite--dimensional underlying space $\mathbf M$ which need not be separately discussed here. Let us only recall the familiar Frobenius theorem which ensures that then the diffiety $\Omega\subset\Phi(\mathbf M)$ has a~basis $df^1,\ldots,df^\mu$ and the solution depends on $\mu=\dim\Omega=e_{-1}$ parameters. We formally put $\nu\ (=\nu(\Omega))=-1.$

\subsection*{(IV) On the symmetries \cite{T6, T7}}\label{ssec2.4}
Admissible mappings $\mathbf m:\mathbf M\rightarrow\bar{\mathbf M}$ between manifolds satisfy $\mathbf m^*\mathcal F(\bar{\mathbf M})\subset\mathcal F(\mathbf M)$ whence $\mathbf m^*\Phi(\bar{\mathbf M})\subset\Phi(\mathbf M).$
\begin{definition}\label{def2.4} Let $\Omega\subset\Phi(\mathbf M)$ and $\bar\Omega\subset\Phi(\bar{\mathbf M})$ be diffieties. Then the mapping $\mathbf m:\mathbf M\rightarrow\bar{\mathbf M}$ is said to be a~\emph{morphism of diffieties} if $\mathbf m^*\bar\Omega\subset\Omega.$ Invertible morphisms are \emph{isomorphisms} of diffieties. Assuming $\mathbf M=\bar{\mathbf M}$ and $\Omega=\bar\Omega,$ invertible morphisms are called \emph{symmetries} (or: \emph{automorphisms}).
\end{definition}
Three subcases of symmetries can be informally mentioned as follows. First, if a~\emph{given} filtration is preserved (Figure~1a). Second, if an~\emph{unknown} filtration is preserved (Figure~1b). Third, if no \emph{finite--dimensional subspace} of $\Omega$ is preserved (Figure~1c).

\medskip

\begin{picture}(120,100)(0,-30)
\put(10,50){\line(1,0){10}}\put(20,50){\line(0,-1){55}}\put(5,55){$\Omega_0$}
\qbezier(20,40)(-10,25)(20,10)\put(18,11){\vector(2,-1){1}}
\put(40,50){\line(1,0){10}}\put(50,50){\line(0,-1){55}}\put(40,55){$\Omega_1$}
\qbezier(50,40)(20,25)(50,10)\put(48,11){\vector(2,-1){1}}
\put(70,50){\line(1,0){10}}\put(80,50){\line(0,-1){53}}\put(70,55){$\Omega_2$}
\qbezier(80,40)(50,25)(80,10)\put(78,11){\vector(2,-1){1}}
\put(95,55){$\ldots$}
\put(45,-15){(1a)}
\end{picture}
\begin{picture}(120,100)(0,-30)
\put(10,50){\line(1,0){10}}\put(20,50){\line(0,-1){55}}\put(5,55){$\Omega_0$}
\qbezier[50](10,10)(44,29)(78,48)\put(10,10){\line(-1,1){10}}
\qbezier(60,38)(37,55)(30,21)\put(30.5,23){\vector(-1,-2){1}}
\put(40,50){\line(1,0){10}}\put(50,50){\line(0,-1){55}}\put(40,55){$\Omega_1$}
\qbezier[50](30,-5)(66,15)(96,33)\put(30,-5){\line(-1,1){10}}
\qbezier(85,26)(62,45)(55,9)\put(55,10){\vector(-1,-2){1}}
\put(70,50){\line(1,0){10}}\put(80,50){\line(0,-1){53}}\put(70,55){$\Omega_2$}
\put(95,55){$\ldots$}
\put(45,-15){(1b)}
\put(30,-35){Figure 1.}
\end{picture}
\begin{picture}(100,100)(0,-30)
\put(10,50){\line(1,0){10}}\put(20,50){\line(0,-1){55}}\put(5,55){$\Omega_0$}
\qbezier[50](10,10)(44,29)(78,48)\put(10,10){\line(-1,1){10}}
\qbezier(20,30)(9,35)(13,13)\put(13,12.5){\vector(0,-1){1}}
\put(40,50){\line(1,0){10}}\put(50,50){\line(0,-1){55}}\put(40,55){$\Omega_1$}
\qbezier(50,20)(39,25)(43,3)\put(43,2.5){\vector(0,-1){1}}
\qbezier[50](30,-5)(66,15)(96,33)\put(30,-5){\line(-1,1){10}}
\put(65,55){$\ldots$}
\put(35,-15){(1c)}
\end{picture}

\bigskip\medskip

The common methods of the general equivalence \cite{T17} can be directly applied only to the subcase (1a) and with slight adaptations even to the subcase (1b) not occuring in the classical theory. Alas, the common methods fail in the subcase (1c). No universal finite algorithm for the determination of \emph{all} symmetries or equivalences of diffieties is known.

\subsection*{(V) On the variations \cite{T6, T7}}\label{ssec2.5}
We expect that the determination of ``approximative symmetries'' is easier. They are realized by vector fields.
\begin{definition}\label{def2.5} \emph{Variations} $Z\in\mathcal T(\mathbf M)$ of a~diffiety $\Omega$ are defined by the condition $\mathcal L_Z\Omega\subset\Omega.$ Variations generating a~Lie group are called (\emph{infinitesimal}) \emph{symmetries}.
\end{definition}
The Figure~1 with arrows denoting the Lie derivative $\mathcal L_Z$ can be mentioned as well, however, the comments are quite other than in the previous case of the symmetries~$\mathbf m.$ In more detail, the subcases (1a) and (1b) concern the variations $Z$ which generate a~group, that is, we have the infinitesimal symmetries. Then, a~somewhat paradoxically, the ``true variations'' of the subcase (1c) do not cause more difficulties.
\begin{lemma}\label{l2.1} A~variation $Z$ is characterized by the property
\begin{equation}\label{eq2.3} (\mathcal L_D\omega)(Z)=D\omega(Z)\qquad (D\in\mathcal H, \omega\in\Omega).
\end{equation}
\end{lemma}
\begin{proof}
Inclusion $\mathcal L_Z\Omega\subset\Omega$ is equivalent to the identity $\mathcal L_Z\mathcal H=[Z,\mathcal H]\subset\mathcal H$ which follows from the equations
\[0=\Omega(\mathcal H),\ 0=Z(\Omega(\mathcal H))=\mathcal L_Z\Omega(\mathcal H)+\Omega([Z,\mathcal H]).\]
So we have \\
\[D(\omega(Z))=\mathcal L_D\omega(Z)+\omega([D,Z])=\mathcal L_D\omega(Z).\]\end{proof}
This simple Lemma~\ref{l2.1} involves the common \emph{linearization procedure} \cite{T13,T14} as a~particular subcase when $\omega=\omega\{f\}.$ If a~\emph{variation} $Z$ is represented by the series (\ref{eq2.1}) then the coefficients $\varphi^j(Z)$ with appropriately chosen forms $\varphi^j\in\Omega$ can be effectively described.
On the contrary, the study of the \emph{symmetries} $Z$ is rather difficult: they satisfy one additional condition, the invariance of some filtration. In general, there are too many filtrations and this prevents us from resolving the symmetry problem completely.

\subsection*{(VI) On the evolutional diffieties}\label{ssec2.6}
The infinitesimal symmetry $Z\in\mathcal T(\mathbf M)$ of a~diffiety $\Omega\subset\Phi(\mathbf M)$ is a~classical concept with simple geometrical significance, the flow on the underlying space~$\mathbf M.$ The variations $Z$ look rather ambiguously in this respect, they are rather arguably identified with virtual flows on the vague space of solutions of diffiety $\Omega$ \cite{T13, T14}.
A~rigorous view is however possible \cite{T6}.

Let us introduce the direct product $\bar{\mathbf M}=\mathbf M\times\mathbb R$ of manifolds with coordinate~$t$ in the factor $\mathbb R.$ Omitting the technicalities, a~function $f\in\Phi(\mathbf M)$ can be regarded as a~function on~$\bar{\mathbf M}$ (independent of $t$) and analogously for the forms $\varphi\in\Phi(\mathbf M).$ If $\varphi^1,\varphi^2,\ldots$ is a~basis of module $\Phi(\mathbf M)$ then $dt,\varphi^1,\varphi^2,\ldots$ constitute a~basis of $\Phi(\bar{\mathbf M}).$ 
The vector fields $\bar Z\in\mathcal T(\bar{\mathbf M})$ can be described as follows.
There are \emph{horizontal} vector fields $H\in\mathcal T(\bar{\mathbf M})$ satisfying $Ht=0$ and they may be identified with vector fields $H(t)\in\mathcal T(\mathbf M)$ depending on parameter~$t$ which are distributed over $\bar{\mathbf M}$  by the $t$--shifts along the $\mathbb R$ component.
Then, by using the obvious \emph{vertical} vector field $\partial/\partial t,$ we obtain the unique decomposition
\[\bar Z=H+\bar f\frac{\partial}{\partial t}\quad (\bar f=\bar Zt, Ht=0, Hf=H(t)f \text{ where }f\in\mathcal F(\mathbf M)\subset\mathcal F(\bar{\mathbf M}))\] into the horizontal and the vertical summands.

With this preparation, the following diffieties of rather special kind provide the rigorous geometrical sense of the variations.

\begin{definition}\label{def2.6} Let $Z(t)\in\mathcal T(\mathbf M)$ be a~variation depending on parameter~$t$ of a~diffiety $\Omega\subset\Phi(\mathbf M).$ We introduce the \emph{evolutional diffiety} $\bar\Omega\subset \Phi(\mathbf M\times\mathbb R)$ with generators
\begin{equation}\label{eq2.4} \omega(Z(t))dt-\omega\in\bar\Omega\subset\Phi(\bar{\mathbf M})\qquad (\omega\in\Omega).\end{equation}
Alternatively, the module $\mathcal H(\bar\Omega)$ is generated by the vector fields
\begin{equation}\label{eq2.5} E=Z(t)+\frac{\partial}{\partial t}\in\mathcal H(\bar\Omega)\subset\mathcal T(\bar{\mathbf M})\text{ and } D\in\mathcal H(\Omega)\end{equation}
where the vector fields $D\in\mathcal H(\Omega)$ are identified with horizontal vector fields distributed over $\mathbf M\times\mathbb R.$ \end{definition}
\begin{definition}\label{def2.7} 
An~inclusion $\mathbf n:\mathbf N\rightarrow\mathbf M$ $(\mathbf N\subset\mathbb R^n, n=n(\Omega))$ is said to be a~\emph{solution} of diffiety $\Omega\subset\Phi(\mathbf M)$ if $\mathbf n^*\Omega=0.$
\end{definition}
In accordance with tradition, we informally identify $\mathbf N=\mathbf n\mathbf N$ which better corresponds to the intuition: the diffiety $\Omega$ identically vanishes on the subspace $\mathbf N=\mathbf n\mathbf N\subset\mathbf M$ of the total space or, equivalently, all vector fields $D\in\mathcal H(\Omega)$ are tangent to this subspace $\mathbf N=\mathbf n\mathbf N\subset\mathbf M$ of the dimension $n=n(\Omega).$ We recall that the existence of solutions is a~highly delicate task \cite{T12,T11,T16}, to put it mildly.

\bigskip

The resulting point is as follows. Let $\bar{\mathbf N}\subset\bar{\mathbf M}$ be a~solution of the evolutional diffiety $\bar\Omega,$ hence $\dim\bar{\mathbf N}=\dim\mathcal H(\bar\Omega)$ and vector fields $\bar D\in\mathcal H(\bar\Omega)$ are tangent to $\bar{\mathbf N}.$ Let $\mathbf N\subset\mathbf M$ be the projection of $\bar{\mathbf N}.$ Since the vector fields $D$ distributed over $\bar{\mathbf M}$ are tangent to $\bar{\mathbf N},$ we conclude that the projections of the level subsets $t=const.$ of $\bar{\mathbf N}$ on $\mathbf N$ are solutions of $\Omega.$ On the other hand, $E$ is tangent to $\bar{\mathbf N}$ as well and generates a~one--parameter group on $\bar{\mathbf N}$ where the $\mathbb R$--component involves mere translations $t\rightarrow t+c.$ It follows that the level sets on $\bar{\mathbf N}$ are permuted and the projections of the level sets in $\mathbf N,$ the solutions of $\Omega,$ are permuted as well. 

We summarize: a~variation (possibly depending on a~parameter) generates many flows, but each only on a~rather narow family of solutions of diffiety $\Omega.$
\begin{prop}\label{th2.1} Let $\bar{\mathbf N}\subset\bar{\mathbf M}=\mathbf M\times\mathbb R$ be a~solution of the evolutional diffiety $\bar\Omega\subset\Phi(\bar{\mathbf M)}$ and $\mathbf N\subset\mathbf M$ the natural projection of $\bar{\mathbf N}.$ Then the level subsets $t=const.$ of $\bar{\mathbf N}$ are projected on the solutions of diffiety $\Omega\subset\Phi(\mathbf M).$ The vector field $E\in\mathcal T(\bar{\mathbf M})$ generates a~Lie group on $\bar{\mathbf N}$ and its projection $Z(t)\in\mathcal T(\mathbf M)$ permutes the solutions of $\Omega$ contained in $\mathbf N.$
\end{prop}
The multi--parameter evolution diffieties for the case of a~\emph{finite--dimensional} Lie algebra of variations $Z$ can be analogously introduced as well.

\section{The involutiveness}\label{sec3}
In the classical external theory, the involutivity ensures that a~given finite system of differential equations is compatible. In the classical internal theory, the involutivity ensures the same for a~finite Pfaffian system. We are interested in diffieties~$\Omega$ where the compatibility is already attained. Then the involutivity describes the structure of the higher--order summands of the $\mathcal A$--module $\mathcal M$ corresponding to a~\emph{given good filtration}~$\Omega_*$ and this is a~pure algebra.

Let $Z_1,\ldots, Z_n$ $(n=n(\Omega))$ be a~basis of module $\mathcal H\, (=\mathcal H(\Omega))$ and $\mathcal A(i)\subset\mathcal A$ $(i=0,\ldots,n)$ the ideal generated by  $Z_1,\ldots, Z_i.$ In particular $\mathcal A(0)=0$ and
\[\mathcal A(n)=\mathcal A_1\oplus\mathcal A_2\oplus\cdots=\mathfrak m\subset\mathcal A\qquad (\mathcal A_1=\mathcal H, \mathcal A_2=\mathcal H\odot\mathcal H, \ldots)\] is the \emph{maximal ideal}. We introduce the factormodules
\[\mathcal M(i)=\mathcal M/\mathcal A(i)\mathcal M=\mathcal M(i)_0\oplus\mathcal M(i)_1\oplus\cdots\quad (\mathcal M(i)_l=\mathcal M_l/\mathcal A(i)\mathcal M\cap\mathcal M_l)\] which are $\mathcal A$--modules as well. In particular
\begin{equation}\label{eq3.1} Z_{i+1}:\mathcal M(i)_l\rightarrow\mathcal M(i)_{l+1}\qquad (i=0,\ldots,n-1).\end{equation}
\begin{definition}\label{def3.1} The basis $Z_1,\ldots, Z_n$ is called \emph{ordinary} (or: \emph{generic} \cite{T6}) for a~given good filtration $\Omega_*$ if (\ref{eq3.1}) are injections for $l$ large enough. \end{definition}
\begin{theorem}[\cite{T6,T8}]\label{th3.1} There exists the ordinary basis of $\mathcal H.$ \end{theorem}
In full generality, this is rather nontrivial result, however, for all current examples, such a~basis can be easily found: the vector fields $Z_1,Z_2,\ldots$ should not be too special, see also below. We are passing to the main topic.
\begin{definition}\label{def3.2} The basis $Z_1,\ldots, Z_n$ of $\mathcal H$ is called \emph{quasiregular} if (\ref{eq3.1}) are injections for all $l\geq 1.$ A~filtration $\Omega_*$ is called \emph{involutive} if there exists a~quasiregular basis and moreover $\mathcal H\mathcal M_l=\mathcal M_{l+1}$ $(l\geq 0).$ \end{definition}
The last condition is clearly equivalent to $\Omega_l+\mathcal L_\mathcal H\Omega_l=\Omega_{l+1}$ $(l\geq 0)$ which can be ensured by a~simple change of the original filtration. Let us introduce the $c$--lift $\Omega_{*+c}$ $(c=0,1,\ldots)$ of the original filtration $\Omega_*$ such that
\[\Omega_{*+c}=\bar\Omega_*:\bar\Omega_0=\Omega_c\subset\bar\Omega_1=\Omega_{c+1}\subset\cdots\subset\Omega=\cup\bar\Omega_l=\cup\Omega_{l+c}.\] Due to the second requirement (\ref{eq2.1}), it follows easily that Theorem~\ref{th3.1} is equivalent to
\begin{theorem}\label{th3.2} Every lift $\Omega_{*+c}$ with $c$ large enough is involutive. \end{theorem}
This provides the essence of all prolongations into the involutivity mechanisms \cite{T12,T17,T6,T18}.
Alas the singular solutions can be included, see below.

\bigskip

We conclude with a~brief description of the ordinary basis. The first term $Z_1$ appears as follows \cite{T8}. There is a~finite family of prime ideals $\mathfrak p\subset\mathcal A, \mathfrak p\neq\mathcal A,$ \emph{associated} to the module $\mathcal M,$ that is, such that there exists a~submodule of $\mathcal M$ isomorphic to $\mathcal A/\mathfrak p.$ The multiplications $Z:\mathcal M_l\rightarrow\mathcal M_{l+1}$ $(Z\in\mathcal H)$ are all injective if and only if
\begin{equation}\label{eq3.2} Z\notin\mathcal A_1\cap (\cup\mathfrak p)=\mathcal H\cap (\cup\mathfrak p).\end{equation}
If all $\mathfrak p\neq\mathfrak m$ then such $Z=Z_1$ does exist. If however $\mathfrak m$ belongs to the associated ideals then the submodules of $\mathcal M$ isomorphic to $\mathcal A/\mathfrak m=\mathcal F$ can be deleted if the original filtration $\Omega_*$ is replaced with a~$c$--lift large enough. The following terms $Z_2,\ldots,Z_n$ appear by analogous construction applied to the $\mathcal A$--modules $\mathcal M(1),\ldots,\mathcal M(n-1)$ instead of $\mathcal M=\mathcal M(0).$

Altogether we conclude that the terms $Z_1,\ldots, Z_n$ of the ordinary basis should not lie in a~finite family of certain linear subspaces of the $\mathcal F$--module $\mathcal H.$ 
We will succintly express this property by saying that they are ``not too special''.
It follows that the total derivatives $D_1,\ldots,D_n\in\mathcal H$ are ``not too special'' for an~appropriate ``not too special'' choice of the independent variables $x_1,\ldots,x_n.$

\bigskip

{\sl Remark. }The primary approach to the compatibility of the systems of differential equations (\ref{eq1.9}) directly use the commutativity $\partial^2/\partial x\partial y=\partial^2/\partial y\partial x$ of various second derivatives which results in perfect ultimate theory \cite{T15}. Alas, though this theory can be effectively applied to particular problems, the calculations strongly depend on subtle formal details. On the contrary, the \'E. Cartan's involutivity \cite{T16} subsequently completed with the prolongation procedure \cite{T17} is of the genuinely geometrical nature. In the actual rigorous expositions, this topic however belongs to the most difficult tasks even though the powerful tools of the commutative and the homological algebra are applied \cite{T12,T18}.
The classical involutivity concept differs from ours in Definition~\ref{def3.2}, since even the involutivity of a~finite--order system of differential equations and of a~finite Pfaffian system is introduced \cite{T12}--\cite{T17},\cite{T6}. It~follows that Theorem~\ref{th3.2} declares the involutivity of every Pfaffian system $\omega=0$ $(\omega\in\Omega_l, l$ fixed and large enough) in the common classical sense.

\section{The standard filtrations}\label{sec4}
Theorems~\ref{th3.1} and \ref{th3.2} concern the higher--order terms $\mathcal M[i]_l$ of the $\mathcal A$--module $\mathcal M[i].$ Returning to the original filtration $\Omega_*,$ they describe a~certain property of the higher--order terms $\Omega_l$ of a~good filtration $\Omega_*.$

For given vector fields $Z_1,Z_2,\ldots\in\mathcal H,$ let us introduce the large series of accompanying ``rough'' filtrations
\[\begin{array}{ll}
\Omega(Z_1)_*:\Omega(Z_1)_0\subset\Omega(Z_1)_1\subset\ldots\subset\Omega&(\Omega(Z_1)_l=\sum\mathcal L^k_{Z_1}\Omega_l),\\
\Omega(Z_1,Z_2)_*:\Omega(Z_1,Z_2)_0\subset\Omega(Z_1,Z_2)_1\subset\ldots\subset\Omega& (\Omega(Z_1,Z_2)_l=\sum\mathcal L^k_{Z_2}\Omega(Z_1)_l),\\
\cdots \end{array}\] of diffiety $\Omega.$ For every submodule $\Theta\subset\Phi$ and a~vector field $Z\in\mathcal H(\Theta)$ we moreover introduce the submodule $\text{Ker}_Z\Theta\subset\Theta$ of all $\vartheta\in\Theta$ with $\mathcal L_Z\vartheta\in\Theta.$ (The latter concept will be applied only in the particular case when $\Theta\subset\Omega$ and $Z\in\mathcal H=\mathcal H(\Omega)\subset\mathcal H(\Theta).$) One can see that Theorems~\ref{th3.1} and \ref{th3.2} are equivalent to the equalities
\begin{equation}\label{eq4.1}\begin{array}{c}
\mbox{Ker}_{Z_1}\Omega_{l+1}=\Omega_l,\\ \mbox{Ker}_{Z_2}\Omega(Z_1)_{l+1}=\Omega(Z_1)_l,\\ \mbox{Ker}_{Z_3}\Omega(Z_1,Z_2)_{l+1}=\Omega(Z_1,Z_2)_l,\\ \ldots \end{array}
\end{equation}
valid for $l$ large enough and not too special $Z_1,Z_2,\ldots\in\mathcal H.$
Our next aim is to ensure (\ref{eq4.1}) for all values of $l$ after appropriate adjustements and this is possible if certain obstructions $\mathcal R^0,\mathcal R^1,\ldots$ are absent.

Let us start with the \emph{first equality} (\ref{eq4.1}). Abbreviating $X=Z_1,$ we may suppose $\mbox{Ker}_X\Omega_l=\Omega_{l-1}\ (l\geq L)$ and consider the inclusions
\[\cdots\supset\Omega_L=\mbox{Ker}_X\Omega_{L+1}\supset\Omega_{L-1}=\mbox{Ker}_X\Omega_L \supset\mbox{Ker}_X^2\Omega_L\supset\mbox{Ker}_X^3\Omega_L\supset\cdots.\]
The strict inclusions terminate with certain equalities $\mbox{Ker}_X^k\Omega_L=\mbox{Ker}_X^{k+1}\Omega_L$ $(k\geq K)$ and so we obtain the new filtration
\begin{equation}\label{eq4.2}\begin{array}{l}\bar\Omega_*:\bar\Omega_0=\mbox{Ker}_X^{K-1}\Omega_L\subset\bar\Omega_1=\mbox{Ker}_X^{K-2}\Omega_L \subset\cdots\\ \subset\bar{\Omega}_{K-2}=\mbox{Ker}_X\Omega_L\subset\bar\Omega_{K-1}=\Omega_L\subset\bar\Omega_K=\Omega_{L+1}\subset\cdots
\end{array}\end{equation}
of diffiety $\Omega$ with strict inclusions together with the submodule
\[\mathcal R^0=\mbox{Ker}_X^K\Omega_L=\mbox{Ker}_X\bar\Omega_0\subset\bar\Omega_0\subset\Omega.\]
\begin{theorem}[\cite{T6,T8,T9}]\label{th4.1}
Filtration $(\ref{eq4.2})$ does not depend on the choice of $X.$ The module $\mathcal R^0$ is flat and does not depend even on the choice of the original good filtration $\Omega_*.$ \end{theorem}
So we may denote $\mathcal R^0=\mathcal R^0(\Omega)$ and there are equalities
\[\mbox{Ker}_X\bar\Omega_{l+1}=\bar\Omega_l\ (l>0), \mbox{Ker}_X\bar\Omega_0=\mathcal R^0, \mbox{Ker}_X\mathcal R^0=\mathcal R^0\]
corresponding to the injections
\[X:\bar{\mathcal M}_l\rightarrow\bar{\mathcal M}_{l+1}\ (l>0),\ \bar{\mathcal M}_0/\mathcal R^0\rightarrow\bar{\mathcal M}_1\ (\bar{\mathcal M}_l=\bar\Omega_l/\bar\Omega_{l-1})\]
for all not too special vector fields $X\in\mathcal H.$ The first equality (\ref{eq4.1}) is universal if and only if $\mathcal R^0=0$ is the trivial module.

Passing to the \emph{second equality} (\ref{eq4.1}), the resonings will be applied ``modulo $Z_1$'' as follows. We abbreviate $Y=Z_2$ and consider the inclusions
\[\cdots\supset\Omega(X)_L\supset\Omega(X)_{L-1}=\mbox{Ker}_Y\Omega(X)_L\supset\mbox{Ker}_Y^2\Omega(X)_L \supset\mbox{Ker}_Y^3\Omega(X)_L\supset\cdots\]
with $L$ large enough. The strict inclusions again terminate and we obtain the new filtration
\begin{equation}\label{eq4.3}\begin{array}{l}\bar\Omega(X)_*:\bar\Omega(X)_0=\mbox{Ker}_Y^{K-1}\Omega(X)_L\subset\bar\Omega(X)_1=\mbox{Ker}_Y^{K-2}\Omega(X)_L \subset\cdots\\ \subset\bar{\Omega}(X)_{K-2}=\mbox{Ker}_Y\Omega(X)_L\subset\bar\Omega(X)_{K-1}=\Omega(X)_L\subset\bar\Omega(X)_K=\Omega(X)_{L+1}\subset\cdots
\end{array}\end{equation}
of diffiety $\Omega$ with strict inclusions together with the module
$$\mathcal R^1=\mbox{Ker}_Y^K\Omega(X)_L=\mbox{Ker}_Y\bar\Omega(X)_0\subset\bar\Omega(X)_0\subset\Omega.4$$
\begin{theorem}[\cite{T8,T9}]\label{th4.2}
Filtration $(\ref{eq4.3})$ does not depend on the choice of $Y.$ The module $\mathcal R^1$ is flat and does not depend even on the choice of the original good filtration $\Omega_*.$ \end{theorem}
We may denote $\mathcal R^1=\mathcal R^1(\Omega)$ and there are equalities
\[\mbox{Ker}_Y\bar\Omega(X)_{l+1}=\bar\Omega(X)_l\ (l>0),\ \mbox{Ker}_Y\bar\Omega(X)_0=\mathcal R^1,\   \mbox{Ker}_Y\mathcal R^1=\mathcal R^1\]
corresponding to the injections
\[Y:\bar{\mathcal M}[1]_l\rightarrow\bar{\mathcal M}[1]_{l+1}\ (l>0),\ \bar{\mathcal M}[1]_0/\mathcal R^1\rightarrow\bar{\mathcal M}[1]_1\ (\bar{\mathcal M}[1]_l=\bar\Omega(X)_l/\bar\Omega(X)_{l-1})\]
for all not too special vector fields $Y\in\mathcal H.$ The second equality (\ref{eq4.1}) is universal if and only if $\mathcal R^1=0$ is the trivial module.
 
 \medskip
 
The procedure can be continued with $Z_3,Z_4,\ldots$ as well with quite analogous result. We obtain certain filtrations
\[\bar\Omega_*,\bar\Omega(Z_1)_*,\bar\Omega(Z_1,Z_2)_*,\ldots\]
which are \emph{good} in the common sense that
\begin{equation}\label{eq4.4}\mathcal L_{\mathcal H}\bar\Omega(\cdot)_l\subset\bar\Omega(\cdot)_l\ (\emph{all }l),\quad \bar\Omega(\cdot)_l+\mathcal L_{\mathcal H}\bar\Omega(\cdot)_l=\bar\Omega(\cdot)_{l+1}\ (l\mbox{ large enough})\end{equation}
and moreover \emph{standard} in the sense that
\begin{equation}\label{eq4.5}\begin{array}{l}
\emph{Ker}_{Z_{k+1}}\bar\Omega(Z_1,\ldots,Z_k)_{l+1}=\bar\Omega(Z_1,\ldots,Z_k)_{l}\quad (l>0),\\
\emph{Ker}_{Z_{k+1}}\bar\Omega(Z_1,\ldots,Z_k)_0=\mathcal R^k,\ \emph{Ker}_{Z_{k+1}}\mathcal R^k=\mathcal R^k.
\end{array}\end{equation}
The procedure becomes trivial if $k>\nu=\nu(\Omega)$ since then
\[\Omega(Z_1,\ldots,Z_k)_l=\Omega\ (l\mbox{ large enough}), \ \bar\Omega(Z_1,\ldots,Z_k)_l=\Omega,\ \mathcal R^k=\Omega.\]
The resulting \emph{residual submodules} $\mathcal R^k\subset\Omega$ constitute the \emph{controllability series}
\begin{equation}\label{eq4.6}\mathcal R^0\subset\mathcal R^1\subset\cdots\subset\mathcal R^{\nu+1}=\Omega,\end{equation} to be discussed below in more detail. We speak of a~\emph{controllable diffiety} $\Omega$ if $\mathcal R^0=\cdots =\mathcal R^\nu=0$ are trivial modules.

\section{Ordinary differential equations}\label{sec5}
We interrupt the general theory for a~relax. Let us mention the relatively simple diffieties $\Omega\subset\Phi(\mathbf M)$ with one independent variable $x=x_1$ (abbreviation). Omitting the trivial subcase $\dim\mathbf M<\infty,$ we have the Hilbert polynomial
\begin{equation}\label{eq5.1}
\dim\mathcal M_l=c_0=\mu(\Omega)>0\qquad (l\text{ large enough}). \end{equation}
The involutivity becomes trivial, however, the standard filtration (\ref{eq4.2}) with any nonvanishing vector field $X=Z_1\in\mathcal H(\Omega)$ is worth mentioning.

We may choose $X=D=D_1$ the formal derivative. Let us moreover suppose
\[\text{Ker}_D\Omega_{l+1}=\Omega_l\quad (l\geq L),\qquad \text{Ker}^k_D\Omega_L=\text{Ker}^{k+1}_D\Omega_L\quad (k\geq K).\]
Then the standard filtration
\begin{equation}\label{eq5.2}
\bar\Omega_*:\bar\Omega_0=\text{Ker}^{K-1}_D\Omega_L\subset\cdots\subset\bar\Omega_{K-2}=\text{Ker}_D\Omega_L\subset
\bar\Omega_{K-1}=\Omega_L\subset\bar\Omega_K=\Omega_{L+1}\subset\cdots
\end{equation}
together with the flat submodule
\begin{equation}\label{eq5.3}
\mathcal R^0=\text{Ker}^K_D\Omega_L=\text{Ker}_D\bar\Omega_0\subset\Omega_0\qquad (\text{Ker}_D\mathcal R^0=\mathcal R^0)\end{equation}
easily appear by a~merely linear algebra.
In the meantime, we also obtain a~rather useful basis of diffiety $\Omega$ as a~by--product. This is made as follows.
Let us choose a~basis\\\\
$\tau^r$ $(r=1,\ldots,R=\dim\mathcal R^0)$ of module $\mathcal R^0$ and then together with\\
$\pi^j_0$ $(j=1,\ldots,j_0)$ basis of the module $\bar\Omega_0$ and then together with\\
$\pi^j_1=\mathcal L_D\pi^j_0$ ($j$ as above), $\pi^{j'}_0$ $(j'=j_0+1,\ldots,j_1)$ of module $\bar\Omega_1,$ and then with\\ $\pi^j_2=\mathcal L^2_D\pi^j_0, \pi^{j'}_1=\mathcal L_D\pi^{j'}_0,$ $\pi^{j''}_0$ $(j''=j'_1+1,\ldots,j'_2)$ of module $\bar\Omega_2,$
$$\cdots\ .$$
The procedure in a~certain sense stops. The identity
\[\text{Ker}_D\bar\Omega_K=\text{Ker}_D\Omega_{L+1}=\Omega_L=\bar\Omega_{K-1}\]
implies $j_K=j_{K+1}$ and analogously $j_{K+1}=j_{K+2}=\cdots$ as well. We obtain only a~finite number $j_k$ of \emph{initial forms}
\begin{equation}\label{eq5.4}
\pi^1_0,\ldots,\pi^{j_0}_0\in\bar\Omega_0;\pi^{j_0+1},\ldots,\pi^{j_1}_0\in\bar\Omega_1;\ldots; \pi^{j_{K-1}+1}_0,\ldots,\pi^{j_K}_0\in\bar\Omega_K\end{equation}
with the lower zero indices and they provide the so called \emph{standard basis}
\begin{equation}\label{eq5.5}
\tau^r\quad (r=1,\ldots,R),\quad\pi^j_s=\mathcal L^s_D\pi^j_0\quad (j=1,\ldots,j_K;\, s=0,1,\ldots\, )\end{equation}
of diffiety $\Omega.$ In fact $j_K=c_0=\mu(\Omega)$ follows from (\ref{eq5.1}) and since $\mathcal R^0$ is flat, there exist even a~basis $\tau^r=dt^r$ $(r=1,\ldots,R=\dim\mathcal R^0).$

The result can be transparently visualized (see Figure~2): the original ``cross--arrows $\mathcal L_Z$'' are ``collected'' in $\mathcal R^0$ and only the infinite sequences $\pi^j_r=\mathcal L^r_D\pi^j_0$ $(j=1,\ldots,\mu(\Omega);\, r=0,1,\ldots)$ without any crossing remain.
\begin{center}\begin{picture}(350,135)(0,-60)
\put(10,50){\line(1,0){10}}\put(20,50){\line(0,-1){55}}\put(5,55){$\Omega_0$}
\put(10,35){\circle{2}}\qbezier(10,35)(20,45)(34,35)\put(35,34){\vector(1,-1){1}}
\put(10,20){\circle{2}}\qbezier(10,20)(20,30)(34,20)\put(35,19){\vector(1,-1){1}}
\put(40,50){\line(1,0){10}}\put(50,50){\line(0,-1){55}}\put(40,55){$\Omega_1$}
\put(40,35){\circle{2}}\qbezier(40,35)(50,45)(64,35)\put(65,34){\vector(1,-1){1}}
\put(40,20){\circle{2}}\qbezier(40,20)(50,20)(64,31)\put(65,33){\vector(1,1){1}}
\put(70,50){\line(1,0){10}}\put(80,50){\line(0,-1){53}}\put(70,55){$\Omega_2$}
\put(68,35){\circle{2}}\qbezier(68,35)(80,45)(94,35)\put(95,34){\vector(1,-1){1}}
\put(40,5){\circle{2}}\qbezier(40,5)(50,15)(64,5)\put(65,4){\vector(1,-1){1}}
\put(68,5){\circle{2}}\qbezier(68,5)(80,15)(94,5)\put(95,4){\vector(1,-1){1}}
\put(5,-37){(2a) the original filtration}
\put(95,55){$\ldots$}
\put(160,48){\line(1,0){10}}
\multiput(170,0)(0,5){10}{\put(0,0){\line(0,1){3}}}
\put(155,55){$\mathcal R^0$}
\put(150,30){\circle{2}}\qbezier(150,30)(165,50)(165,30)\put(165,30){\vector(0,-1){1}}
\put(190,50){\line(1,0){10}}\put(200,50){\line(0,-1){50}}\put(190,55){$\bar\Omega_0$}
\put(190,20){\circle{2}}\qbezier(190,20)(200,30)(214,20)\put(215,19){\vector(1,-1){1}}
\put(220,20){\circle{2}}\qbezier(220,20)(230,30)(244,20)\put(245,19){\vector(1,-1){1}}
\put(250,20){\circle{2}}\qbezier(250,20)(260,30)(274,20)\put(275,19){\vector(1,-1){1}}
\put(280,20){\circle{2}}\qbezier(280,20)(290,30)(304,20)\put(305,19){\vector(1,-1){1}}
\put(185,10){$\pi^j_0$}\put(215,10){$\pi^j_1$}
\put(242,-10){$\pi^{j^{''}}_0$}\put(274,-10){$\pi^{j^{''}}_1$}
\put(250,5){\circle{2}}\qbezier(250,5)(260,15)(274,5)\put(275,4){\vector(1,-1){1}}
\put(280,5){\circle{2}}\qbezier(280,5)(290,15)(304,5)\put(305,4){\vector(1,-1){1}}
\put(220,50){\line(1,0){10}}\put(230,50){\line(0,-1){50}}\put(220,55){$\bar\Omega_1$}
\put(250,50){\line(1,0){10}}\put(260,50){\line(0,-1){50}}\put(250,55){$\bar\Omega_2$}
\put(280,50){\line(1,0){10}}\put(290,50){\line(0,-1){50}}\put(280,55){$\bar\Omega_3=\Omega_2$}
\put(175,-37){(2b) the standard filtration}
\put(333,55){$\ldots$}
\put(145,-57){Figure 2.}
\end{picture}\end{center}

\medskip

Some immediate consequences of this construction are as follows.
\subsection{The uniqueness}\label{subsec5.1} We recall that the module $\mathcal R^0$ is unique. The standard filtration (\ref{eq5.2}) is unique if $\mu(\Omega)=1$ and we refer to quite simple proof \cite[Theorem~26]{T7}. This is historically the main result of the stimulating article \cite{T17} where the concept of the absolute theory was introduced for the first time. Close connection to the beautiful but forgotten explicit solvability \emph{Monge problem} \cite{T19,T20,T6} is worth mentioning, too.
\subsection{The morphisms \cite{T7}}\label{subsec5.2} If $\mathbf m:\mathbf M\rightarrow\mathbf M$ is a~morphism of $\Omega$ then
\[D\mathbf m^*x\cdot\mathbf m^*\mathcal L_D\omega=\mathcal L_D\mathbf m^*\omega\quad (\omega\in\Omega)\]
and in terms of the standard basis
\begin{equation}\label{eq5.6}
\mathbf m^*\tau^r\in\mathcal R^0, D\mathbf m^*x\cdot\mathbf m^*\pi^j_{s+1}=\mathcal L_D\mathbf m^*\pi^j_s.\end{equation}
Such a~morphism is even invertible ($\mathbf m$ is a~symmetry) if and only if
\begin{equation}\label{eq5.7}
\mathbf m^*\mathcal R^0=\mathcal R^0, \pi^j_0\in\mathbf m^*\Omega\quad (j=1,\ldots,\mu(\Omega)).\end{equation}
It is however not easy to apply these results effectively in the general equivalence theory.
For instance, the solution of the symmetry problem even for the favourable case $\mu(\Omega)=1$ in the famed article \cite{T21} was not yet undestood in full detail, see the last sentence in the prominent textbook \cite{T22}.
\subsection{The variations \cite{T7}}\label{subsec5.3} On the contrary, the simple explicit formula
\begin{equation}\label{eq5.8}
Z=zD+\sum z^r\frac{\partial}{\partial t^r}+\sum D^sp^j\frac{\partial}{\partial\pi^j_s}\end{equation}
for all variations holds true. Here $z, p^j\in\mathcal F(\mathbf M)$ are arbitrary functions and $z^r=z^r(t^1,\ldots,t^R)$ arbitrary composed functions. This is a~consequence of general Lemma~\ref{l2.1} applied to the standard basis (\ref{eq5.5}).
\subsection{The infinitesimal symmetries \cite{T6,T5,T7}}\label{subsec5.4} 
Variation (\ref{eq5.8}) generates a~Lie group if and only if it preserves an~appropriate good (equivalently: standard) filtration. This is informally described in Figures 1a and 1b where the ``dotted filtration'' is not known in advance. In the favourable case $\mu(\Omega)=1,$ we may deal only with the unique standard filtration 1a and the dotted Figure 1b can be omitted.
\subsection{The jet diffieties}\label{subsec5.5}  Passing to more individual examples, we introduce the space $\mathbf M(m)$ with local \emph{jet coordinates}
\begin{equation}\label{eq5.9} x,\ w^j_s\quad (j=1,\ldots,m;\, s=0,1,\ldots\,)\end{equation}
and the submodule $\Omega(m)\subset\Phi(\mathbf M(m))$ of all \emph{contact forms}
\begin{equation}\label{eq5.10} 
\omega=\sum f^j_s\omega^j_s\quad (\text{finite sum}, f^j_s\in\mathcal F(\mathbf M(m)), \omega^j_s=dw^j_s-w^j_{s+1}dx).
\end{equation}
It follows that we have a~diffiety: if $\Omega(m)_l\subset\Omega$ $(l=0,1,\ldots)$ is the submodule of the $l$--th \emph{order contact forms} (with $s\leq l$ in formula (\ref{eq5.10})) then
\begin{equation}\label{eq5.11} 
\Omega(m)_*:\Omega(m)_0\subset\Omega(m)_1\subset\cdots\subset\Omega=\cup\Omega(m)_l\end{equation}
is a good filtration. (See also Preface 1.9 for the choice $n=1$ since then $\mathbf M(m)=\mathbf M(m,1)$ and $\Omega(m)=\Omega(m,1).$) Clearly (\ref{eq5.11}) is even a~standard  filtration where $\mathcal R^0=0$ is trivial module and the contact forms $\omega^j_s=\pi^j_s$ provide the standard basis.

We may refer to \cite{T1} for only a~few particular examples of automorphisms
\mbox{$\mathbf m:\mathbf M(m)\rightarrow\mathbf M(m)$ of diffiety $\Omega(m)$}
involving, e.g., the above wave mechanisms. We recall that they are in general unknown.
On the contrary, all variations
\[Z=zD+\sum D^sp^j\ \frac{\partial}{\partial \omega^j_s}\quad (z,p^1,\ldots,p^m \text{ arbitrary functions})\]
are well--known. They constitute a~huge Lie algebra and one of the main tasks of the soliton theory concerns the determination of the special Abelian Lie subalgebras, the so called \emph{integrable hierarchies}. On this occasion, we cannot forget the impressive monograph \cite{T23}.
\subsection{Example: the Hilbert--Cartan equation}\label{subsec5.6} In order to demonstrate quite explicit results, we mention the infinitesimal symmetries of the differential equation $d^2u/dx^2=F(dv/dx)$ with two unknown functions $u=u(x)$ and $v=v(x)$ thoroughly treated in article \cite{T7}. The equation corresponds to the diffiety $\Omega\subset\Phi(\mathbf M)$ generated by the forms
\[\alpha_0=du_0-u_1dx, \alpha_1=du_1-F(v_1)dx, \beta_r=dv_r-v_{r+1}dx\quad (r=0,1,\ldots)\]
in the space $\mathbf M$ with coordinates $x,u_0,u_1,v_r$ $(r=0,1,\ldots).$ Clearly $\mu(\Omega)=1$ and the standard filtration $\bar\Omega_*$ is unique. Assume $F'\neq 0.$ Then $\mathcal R^0=0$ is trivial and we have only one initial form
\[\pi^1_0=F'\alpha+DF'\alpha_0\quad (D=\frac{\partial}{\partial x}+u_1\frac{\partial}{\partial u_0}+F\frac{\partial}{\partial u_1}+\sum v_{r+1}\frac{\partial}{\partial v_r}).\]
The variations
\[Z=zD+\sum D^rp\frac{\partial}{\partial \pi^1_r}\quad (\text{arbitrary }z,p\in\mathcal F(\mathbf M(m)))\]
preserving moreover the standard filtration are just the infinitesimal symmetries. So we have the requirement
\[\mathcal L_Z\pi^1_0=Z\rfloor d\pi^1_0+d\pi^1_0(Z)=Z\rfloor d\pi^1_0+dp=\lambda\pi^1_0\]
with unknown factor $\lambda.$ The calculations are lengthy. First of all, certain formulae \cite[(107)]{T7} not stated here uniquely express $z$ and $\lambda$ in terms of $p.$ Then the crucial equation
\[p=F'^2P(\cdot)+F'Q(\cdot)v_2\quad ((\cdot)=(x,u_0,u_1,v_0,v_1)),\]
where
\[P=\frac{1}{F'}(Q_x+u_1Q_{u_0})+\int \left(\frac{1}{F'}\right)'Fdv_1\cdot Q_{u_1}+\int \left(\frac{1}{F'}\right)'v_1dv_1\cdot Q_{v_0}+\bar P(x,u_0,u_1,v_0)\]
can be derived, see \cite[(133)]{T7}. In the \emph{generic subcase} \cite[(144)]{T7} we obtain the final solution
\[Q=(Ax+\bar A)u_1+Bv_0+C_1x-2Au_0+C_3, \bar P=Av_0+C\quad (A,\ldots,C\in\mathbb R)\]
but special functions $F$ admit more symmetries. For instance the exceptional 14--dimensional Lie algebra $\mathbb G_2$ of symmetries for the \emph{Hilbert--Cartan equation} where $F=(dv/dx)^{1/2}$ was obtained \cite[(175)]{T7} by direct calculations in full accordance with the article \cite{T21}.
\subsection{Example: the Monge equation}\label{subsec5.7} We mention the differential equation $dw/dx=F(du/dx,dv/dx)$ with nonconstant $F,$ however, only some conceptual topics will be discussed and we refer to \cite{T7} for more detailed survey.
The corresponding diffiety $\Omega\subset\Phi(\mathbf M)$ has the natural basis denoted
\[\alpha_r=du_r-u_{r+1}dx,\beta_r=dv_r-v_{r+1}dx\quad (r=0,1,\ldots), \gamma=dw-F(u_1,v_1)dx\]
in the space $\mathbf M$ with coordinates $x,u_r,v_r$ $(r=0,1,\ldots)$ and $w.$ The forms $\alpha_r,\beta_r$ $(r\leq l)$ and $\gamma$ constitute a~basis of module $\Omega_l$ $(l=0,1,\ldots)$ of the original good filtration $\Omega_*.$ The obvious identity
\[\mathcal L_D\gamma=F_{u_1}\alpha_1+F_{v_1}\beta_1\quad (D=\frac{\partial}{\partial x}+\sum u_{r+1}\frac{\partial}{\partial u_r}+\sum v_{r+1}\frac{\partial}{\partial v_r}+F\frac{\partial}{\partial w}\in\mathcal H)\]
represents the ``cross'' in Figure~2a. Then $\mathcal R^0=0$ and the forms
\[\pi^1_0=\gamma-F_{u_1}\alpha_0-F_{v_1}\beta_0,\ \pi^2_0=\beta_0\]
are initial for the standard filtration of Figure~2b. The alternative basis
\[\pi^1_r=\mathcal L_D^r\pi^1_0\quad (\text{not explicitly stated}),\ \pi^2_r=\mathcal L_D^r\beta_0=\beta_r\quad (r=0,1,\ldots)\]
of diffiety $\Omega$ is better adapted for the calculations than the original one thanks to the absence of ``crossings''.

Let us start with \emph{morphisms} $\mathbf m:\mathbf M\rightarrow\mathbf M$ of diffiety $\Omega.$ They are determined by formulae
\[\mathbf m^*\pi^1_0=\sum a^j_r\pi^j_r,\ \mathbf m^*\pi^2_0=\sum b^j_r\pi^j_r\quad (\text{arbitrary coefficients})\]
since the remaining forms
\[\mathbf m^*\pi^k_r=\mathbf m^*\mathcal L_D^r\pi^k_0\quad (k=1,2;\, r=1,2,\ldots)\]
satisfy the recurrence (\ref{eq5.6}). We are however interested in \emph{invertible morphisms} $\mathbf m$ and then the criterion (\ref{eq5.7}) provides rather strong additional condition for the coefficients $a^j_r$ and $b^j_r.$

On the contrary, the \emph{variations}
\[Z=zD+\sum D^rp\ \frac{\partial}{\partial\pi^1_r}+\sum D^rq\ \frac{\partial}{\partial\pi^2_r}\quad (\text{arbitrary }z,p,q)\]
are given by simple explicit formula. We are however interested in \emph{symmetries}~$Z$ and they moreover are bound to preserve a~certain good filtration (Figures~1a and 1b). Since $\mu(\Omega)=2,$ there are many possibilities and this is the reason why the symmetry problem for our seemingly simple Monge equation cannot be ultimately resolved. We can only refer to three particular examples of symmetries
\[\mathcal L_Z\pi^1_0=\mu\pi^1_0,\ \mathcal L_Z\pi^2_0=\lambda^1_0\pi^1_0+\lambda^2_0\pi^2_0+\lambda^1_1\pi^1_1,\]
\[\mathcal L_Z\pi^1_0=\mu^1_0\pi^1_0+\mu^2_0\pi^2_0+\mu^2_1\pi^2_1,\ \mathcal L_Z\pi^2_0=\mu\pi^2_0,\]
\[\mathcal L_Z\pi^1_0=\lambda^1\pi^1_0+\lambda^2\pi^2_0,\ \mathcal L_Z\pi^2_0=\mu^1\pi^1_0+\mu^2\pi^2_0\]
discussed in \cite{T7} especially for the case $F=u_1v_1.$
\subsection{Continuation: the Monge problem}\label{subsec5.8} There exists an~automorphism $\mathbf m:\mathbf M(3)\rightarrow\mathbf M(3)$ of the jet diffiety $\Omega(3)$ such that
\[\mathbf m^*(x-F(w^1_0,w^2_0))=w^3_1-F(w^1_1,w^2_1)\]
for every function $F,$ see \cite[Appendix]{T7}. Alternatively saying, the family of all curves $w^j_0=w^j_0(x)$ $(j=1,2,3)$ satisfying $x=F(w^1_0,w^2_0)$ is identified with the solutions of the Monge equation
\[\frac{dw}{dx}=F\left(\frac{du}{dx},\frac{dv}{dx}\right)\qquad (w=w^3_0, u=w^1_0, v=w^2_0).\]
Still, in other terms, assume $F\neq const.$ Then we have the subspace $\mathbf N\subset\mathbf M(3)$ given by equations $D^r(x-F)=0$ $(r=0,1,\ldots)$ which is clearly isomorphic to $\mathbf M(2)$ and as a~result, \emph{the corresponding jet diffiety} $\Omega(2)\subset\Phi(\mathbf M(2))$ \emph{is isomorphic to the diffiety} $\Omega\subset\Phi(\mathbf M)$ \emph{of the Monge equation}. It follows that the Monge equation \[\frac{dw}{dx}=F\left(\frac{du}{dx},\frac{dv}{dx}\right)\] can be resolved by certain explicit algebraical formulae involving two arbitrary functions.
\subsection{Calculus of variations \cite{T6}--\cite{T27}, \cite{T7}}\label{subsec5.9} 
The classical Lagrange problem concerning the one--dimensional variational integral subjected to differential constraints is represented by a~diffiety $\Omega\subset\Phi(\mathbf M)$ together with a~given form $\varphi\in\Phi(\mathbf M).$ We are interested in the \emph{variational integrals}
\[\int_a^b\mathbf n^*\varphi\qquad (\mathbf n:\mathbf N\rightarrow\mathbf M,\,\mathbf n^*\Omega=0,\, \mathbf N:a\leq x\leq b\subset\mathbb R)\]
evaluated for the solutions $\mathbf n:\mathbf N\rightarrow\mathbf M$ of diffiety $\Omega.$ Such a~solution is called \emph{extremal} if $\mathbf n^*Z\rfloor d\varphi=0$ for all variations $Z.$ This is in full accordance with the common approach since then the obvious identities
\[\mathbf n^*\mathcal L_Z\varphi=d\,\mathbf n^*\varphi (Z), \int_a^b\mathbf n^*\mathcal L_Z\varphi=\mathbf n^*\varphi (Z)|_{x=a}^{x=b}\] declare that the variation of the integral (in the common sense) indeed depends only on the boundary values. One can observe that the form $\varphi$ can be replaced with any form $\varphi+\omega$ $(\omega\in\Omega)$ without change of the extremals. 
At this place, let us apply the standard filtrations. Assuming the controllability $\mathcal R^0=0,$ a~unique \emph{Poincar\'e--Cartan form} $\varphi+\breve\omega$ with appropriate $\breve\omega\in\Omega$ exists such that
\begin{equation}\label{eq5.12}
d\breve\varphi\cong\sum e^j\pi^j_0\wedge dx\quad (\text{mod }\Omega\wedge\Omega).\end{equation}
This implies that the extremals $\mathbf n$ are characterized by the \emph{Euler--Lagrange equations} $\mathbf n^*e^j=0$ $(j=1,\ldots,\mu(\Omega))$ and satisfy the identity $\mathbf n^*Z\rfloor d\breve\varphi=0$ for \emph{all vector fields} $Z\in\mathcal T(\mathbf M).$

It follows that $e^j=0$ identically if and only if $\breve\varphi$ is a~total differential of appropriate function $g$ hence $\varphi\cong \breve\varphi\cong Dg\cdot dx$ $(\text{mod }\Omega).$
The \emph{Noether theorem} immediately follows as well. Assuming
\begin{equation}\label{eq5.13}\mathcal L_Z\Omega\subset\Omega,\ \mathcal L_Z\varphi\in\Omega\qquad (\text{appropriate }Z\in\mathcal T(\mathbf M)),\end{equation}
the function $\mathbf n^*Z\rfloor d\breve\varphi$ is clearly constant for every extremal $\mathbf n.$
Also the investigation of all symmetries $Z$ of the variational problem (they are defined by properties (\ref{eq5.13})) simplifies \cite{T7}[Sections~7 and 8].
\subsection{Example: a~variational integral}\label{subsec5.10} Explicit formulae are rather clumsy in the general case. So we mention only the variational integral
\[\int f(x,u_0,v_0,w_0,u_1,v_1)dx\qquad (u_r=\frac{d^ru}{dx^r}, v_r=\frac{d^rv}{dx^r})\]
subjected to the constraint \[w_1=F(x,u_0,v_0,w_0,u_1,v_1).\]
Assuming
\[a=F_{v_0}-DF_{v_1}+F_{w_0}F_{v_1}\neq 0\quad (D=\frac{\partial}{\partial x}+\sum u_{r+1}\frac{\partial}{\partial u_r}+\sum v_{r+1}\frac{\partial}{\partial v_r}+F\frac{\partial}{\partial w_0}),\]
there exists the Poincar\'e--Cartan form
\[\breve\varphi=fdx+(f_{u_1}-\frac{b}{a}F_{u_1})\alpha+(f_{v_1}-\frac{b}{a}F_{v_1})\beta-\frac{b}{a}\gamma\]
where
\[b=f_{v_0}-Df_{v_1}+f_{w_0}F_{v_1}, \alpha=du_0-u_1dx, \beta=dv_0-v_1dx, \gamma=dw_0-Fdx\]
and the Euler--Lagrange equations
\[e^1=f_{w_0}-\frac{b}{a}F_{w_0}-D\frac{b}{a}=0,\ e^2=B-\frac{b}{a}A=0\]
where 
\[A=F_{u_0}-DF_{u_1}+F_{w_0}F_{u_1}, \ B=f_{u_0}-Df_{u_1}+f_{w_0}F_{u_1}.\]
Since no uncertain multipliers appear, the Legendre, Jacobi, Hilbert--Weierstrass extremality conditions, the Hamilton--Jacobi equations and the geodesic fields can be investigated without much difficulty \cite{T24}--\cite{T27} quite analogously as in the traditional unconstrained theory.

\end{document}